\documentclass{amsart}
\usepackage{amsfonts,amssymb,amscd,amsmath,enumerate,verbatim,calc}
\usepackage[all]{xy}

\newcommand{\CM}{Cohen-Macaulay}

\newcommand{\wrt}{with respect to}

\newcommand{\n}{\mathfrak{n} }
\newcommand{\m}{\mathfrak{m} }
\newcommand{\M}{\mathfrak{M} }

\newcommand{\A}{\mathfrak{a} }
\newcommand{\R}{\mathcal{R} }
\newcommand{\fp}{\mathfrak{p}}

\newcommand{\fP}{\mathfrak{P}}
\newcommand{\F}{\mathcal{F} }
\newcommand{\eR}{\widehat{\mathcal{R} }}
\newcommand{\Z}{\mathbb{Z} }
\newcommand{\fb}{\mathbf{f} }
\newcommand{\rt}{\rightarrow}

\newcommand{\ov}{\overline}

\newcommand{\wh}{\widehat }
\newcommand{\wt}{\widetilde }

\newcommand{\image}{\operatorname{image}}

\newcommand{\rank}{\operatorname{rank}}
\newcommand{\grade}{\operatorname{grade}}

\newcommand{\Ass}{\operatorname{Ass}}

\newcommand{\Sing}{\operatorname{Sing}}

\newcommand{\Proj}{\operatorname{Proj}}
\newcommand{\Spec}{\operatorname{Spec}}
\newcommand{\embdim}{\operatorname{embdim}}
\newcommand{\projdim}{\operatorname{projdim}}
\newcommand{\height}{\operatorname{height}}
\newcommand{\Fitt}{\operatorname{Fitt}}

\newcommand{\Ext}{\operatorname{Ext}}
\newcommand{\Tor}{\operatorname{Tor}}

\theoremstyle{plain}

\newtheorem{theorem}{Theorem}[section]
\newtheorem{corollary}[theorem]{Corollary}
\newtheorem{lemma}[theorem]{Lemma}
\newtheorem{proposition}[theorem]{Proposition}

\theoremstyle{definition}

\newtheorem{s}[theorem]{}
\newtheorem{remark}[theorem]{Remark}
\newtheorem{example}[theorem]{Example}

\theoremstyle{remark}

\begin{document}

\title[Andr\'e-Quillen homology]{Andr\'e-Quillen homology of Rees algebras and extended Rees algebras}
\author{Tony~J.~Puthenpurakal}
\date{\today}
\address{Department of Mathematics, IIT Bombay, Powai, Mumbai 400 076, India}

\email{tputhen@gmail.com}

 \begin{abstract}
Let $(A,\m)$ be an excellent  local complete intersection ring and let \\ $I = (a_1, \ldots, a_r)$ be an ideal of positive height. Let $\R(I) = A[It]$ be the Rees algebra of $I$. Consider the map $\psi \colon S = A[X_1, \ldots, X_r]  \rt \R(I)$ which maps $X_i \rt a_it$ for all $i$. Let $J = \ker \psi$ and let $H_*(J)$ be the Koszul homology of $J$.  We prove that the following assertions are equivalent:
\begin{enumerate}[\rm (i)]
  \item $\Proj \R(I)$ is a complete intersection.
  \item \begin{enumerate}[\rm (a)]
          \item  $D_3(\R(I)|A, \R(I))_n = 0$ for $n \gg 0$.
          \item For $P \in \Proj \R(I)$ we have $H_1(J)_P$ is a free $\R(I)_P$-module.
        \end{enumerate}
\end{enumerate}
Here $D_3(\R(I)|A, \R(I))$ is the third Andr\'e-Quillen homology of $\R(I)$ with respect to $A \rt \R(I)$. We prove an analogous result for the extended Rees algebra $\eR(I) = A[It, t^{-1}]$.
When $A$ is a \CM \ domain (not necessarily a complete intersection) we compute that rank of $H_1(J)$ and hence compute its free locus.

\end{abstract}
 \maketitle
\section{introduction}
\s \label{setup} Let $(A,\m)$ be a  \CM \  local ring and let $I = (a_1, \ldots, a_r)$ (minimally) be an ideal.  Assume $\height I \geq 1$. Let $\R(I) = A[It]$ be the Rees algebra of $I$. Consider the map $\epsilon \colon S = A[X_1, \ldots, X_r]  \rt \R(I)$ which maps $X_i \rt a_it$ for all $i$. Let $J = \ker \epsilon$. We give the standard grading to $S$. The ideal $J$ is called the defining ideal of the Rees algebra of $I$ and has been extensively studied.   Analogously
let $\eR(I) = A[It, t^{-1}]$ be the extended Rees algebra of $I$. Consider the map $\wh{\epsilon} \colon \wh{S} = A[X_1, \ldots, X_r, T]  \rt \eR(I)$ which maps $X_i \rt a_it$ for all $i$ and $T$ is mapped to $t^{-1}$. Let $\wh{J} = \ker \wh{\epsilon}$. We give the following  grading to $\wh{S}$; set $\deg A = 0$, $\deg X_i = 1$ for all $i$ and $\deg T = -1$.
The ideal $\wh{J}$ is the defining ideal of the extended Rees algebra of $I$. It can be shown that
$\wh{J} = J\wh{S} + (TX_i -a_i \mid 1\leq i \leq r)$, see \cite[5.5.7]{HS}.

Let $H_*(J)$ ($H_*(\wh{J})$) denote the Koszul homology of $J$  (respectively of $\wh{J}$). In this paper we investigate some Koszul homology of $J$ ( respectively of $\wh{J} )$ and decode its impact on properties of $\R(I)$ (respectively of $\eR(I)$).  The main tool for this analysis is the  Andr\'e-Quillen homology $D_*(\R(I)|A, -)$ and $D_*(\eR(I)|A, -)$. We note that Andr\'e-Quillen homology has been earlier studied in the context of Rees algebras by  Andr\'e  \cite{A-m} and by Planas-Vilanova \cite{Pl-96}, \cite{Pl-15}, \cite{Pl-21} and \cite{Pl-22}.

If $T = \bigoplus_{n \in \Z}T_n$ is a graded ring then by $T_+$ we denote the ideal generated by $\bigoplus_{n \geq 1}T_n$.  By $\Proj(T)$ we mean the scheme consisting of homogeneous prime ideals $P$ with $P \nsupseteq T_+$. Andr\'e-Quillen homology is quite effective when dealing with complete intersections. We show

\begin{theorem}\label{proj-ci}
Let $(A,\m)$ be a local complete intersection and let $I$ be an ideal of positive height.
Assume either $A$ is excellent or $A$ is a quotient of a regular local ring.
The following assertions are equivalent:
\begin{enumerate}[\rm (i)]
  \item $\Proj \R(I)$ is a complete intersection.
  \item $\Proj \eR(I)$ is a complete intersection.
  \item \begin{enumerate}[\rm (a)]
          \item $D_3(\R(I)|A, \R(I))_n = 0$ for $n \gg 0$.
          \item   $H_1(J)_P$ is a free $\R(I)_P$-module for  every $P \in \Proj(\R(I))$.
        \end{enumerate}
  \item \begin{enumerate}[\rm (a)]
          \item $D_3(\eR(I)|A, \eR(I))_n = 0$ for $n \gg 0$.
          \item   $H_1(\wh{J})_P$ is a free $\eR(I)_P$-module for every $P \in \Proj(\eR(I))$,.
        \end{enumerate}
\end{enumerate}
\end{theorem}
Bountiful examples of rings satisfying the assumptions of Theorem \ref{proj-ci} arise from Hironaka's resolution of singularities, see \ref{ex-1}.

\begin{remark}
We consider $0$ to be a free module over the ambient ring.
\end{remark}

\s \label{herzog} We note that $D_j(\R(I)|A, -) = D_j(\R(I)|S, -)$ for $j \geq 2$ and $D_j(\eR(I)|A, -) = D_j(\eR(I)|\wh{S},-)$ for $j \geq 2$. If $A$ contains a field of characteristic zero then we have an exact sequence (see \cite[4.6]{H})
$$ D_4(\R(I)|S, \R(I)) \rt      \bigwedge^2H_1(J) \rt H_2(J) \rt    D_3(\R(I)|S, \R(I)) \rt 0,$$
and an exact sequence
$$ D_4(\eR(I)|\wh{S}, \eR(I)) \rt      \bigwedge^2H_1(\wh{J}) \rt H_2(\wh{J}) \rt    D_3(\eR(I)|\wh{S}, \eR(I)) \rt 0.$$
Here $\bigwedge^2H_1(-) \rt H_2(-)$ is the natural multipliction map induced in the algebra $H_*(-)$.
Thus the condition $D_3(\R(I)|A, \R(I))_n = 0$ for $n \gg 0$ (respectively $D_3(\eR(I)|A, \eR(I))_n = 0$ for $n \gg 0$) can be made entirely in terms of Koszul homology of $J$ (and $\wh{J}$ respectively).

In view of Theorem \ref{proj-ci} we might wonder what happens when $\R(I)$ or $\eR(I)$ is a complete intersection. There is a paucity of examples when $\R(I)$ is a complete intersection. However there are bountiful examples of $\eR(I)$ being a complete intersection, see \cite[1.5]{PTJ}. Our result is

\begin{theorem}\label{ex-2-thm} Let $(A, \m)$ be a local complete intersection and let $I$ be an ideal of height $ \geq 1$.
  The following assertions are equivalent:
  \begin{enumerate}[\rm (i)]
    \item $\eR(I)$ is a complete intersection.
    \item $H_1(\wh{J})$ is a free $\eR(I)$-module and $D_3(\eR(I)|A, \eR(I)) = 0$.
  \end{enumerate}
\end{theorem}
In view of the above results it is necessary to understand the free locus of $H_1(J)$ and $H_1(\wh{J})$. When $T$ is local (or *-local) for a finitely generated module $M$ (graded if $T$ is *-local) we denote the number of minimal generators of $M$ by $\mu(M)$.
When $A$ is a domain we prove:
\begin{theorem}\label{rank}
Let $(A,\m)$ be a \CM \ local domain of dimension $d \geq 1$ and let $I$ be a non-zero ideal.
Then we have
\begin{enumerate}[\rm (1)]
\item
$\rank H_1(\wh{J}) = \mu(\wh{J}) - \mu(I)$.
\item
$\rank H_1({J}) = \mu({J}) - \mu(I) + 1$.
\end{enumerate}
\end{theorem}
If $M$ is an $T$-module then let $\Fitt_j(M)$ denote the $j^{th}$ Fitting ideal of $M$.
The following result is well-known and easy to prove:
\begin{proposition}\label{rank-fitt}
Let $T$ be a Noetherian domain and let $M$ be a finitely generated $T$-module of rank $r \geq 1$.
Let $P$ be a prime ideal in $T$. Then the following assertions are equivalent:
\begin{enumerate}[\rm (1)]
\item $M_P$ is free $T_P$-module.
\item $P \nsupseteq \Fitt_r(M)$.
\end{enumerate}
\end{proposition}
By Theorem \ref{rank} and Proposition \ref{rank-fitt} we can determine the free locus of $H_1(\wh{J})$ and $H_1(J)$. When $A$ is a regular local ring we show
\begin{proposition}\label{locus}
Let $A$ be a regular local ring and let $I$ be a non-zero ideal of $A$. We have
\begin{enumerate}[\rm (I)]
  \item Let $P$ be a prime ideal in $\R(I)$. Then the following assertions are equivalent:
  \begin{enumerate}[\rm (a)]
    \item $H_1(J)_P$ is a free $\R(I)_P$-module.
    \item $\R(I)_P$ is a complete intersection.
  \end{enumerate}
  \item Let $P$ be a prime ideal in $\eR(I)$. Then the following assertions are equivalent:
  \begin{enumerate}[\rm (a)]
    \item $H_1(\wh{J})_P$ is a free $\eR(I)_P$-module.
    \item $\eR(I)_P$ is a complete intersection.
  \end{enumerate}
\end{enumerate}
\end{proposition}
Proposition \ref{locus} follows easily  from a result of Gulliksen cf., \cite[1.4.9]{GL}.
We give bounds on the rank of minimal generators of  $H_1(\wh{J})$ when $I$ is equi-multiple
and $\eR(I)$ is a complete intersection. We prove
\begin{theorem}\label{min-gen}
Let $(A,\m)$ be a complete intersection and let   $I$ be an equi-multiple of height $r \geq 1$. Assume   the residue field of $A$ is infinite. Let $Q$ be a minimal reduction of $I$. If $\eR(I)$ is a complete intersection then
we have
$$\mu(H_1(\wh{J})) \leq \embdim A/Q + r - d. $$
\end{theorem}

We give an example (see \ref{hyper}) which shows that equality can occur in the above bound. We also note that $\embdim A/Q \leq \embdim A$. So $\embdim A + r -d$ is an upper bound for minimal number of generators of $H_1(\wh{J})$ which is independent of $I$.

Recall if $T$ is a ring of finite Krull dimension and $E$ is a finitely generated $T$-module then
\begin{enumerate}
  \item $E$ is said to be unmixed if all associate primes of $E$ are minimal.
  \item $E$ is said to be equi-dimensional if $\dim T/P = \dim E$ for all minimal primes $P$ of $E$.
\end{enumerate}
We prove:
\begin{theorem}\label{first}
(with hypotheses as in \ref{setup}).  Further assume $A$ is \CM \ and $\height I > 0$. We have
\begin{enumerate}[\rm (I)]
  \item The following assertions are equivalent:
  \begin{enumerate}[\rm (a)]
    \item $H_1(J)$ is unmixed and equi-dimensional.
    \item $D_2(\R(I)|A, \R(I)) = 0$.
  \end{enumerate}
  \item The following assertions are equivalent:
  \begin{enumerate}[\rm (a)]
    \item $H_1(\wh{J})$ is unmixed and equi-dimensional.
    \item $D_2(\eR(I)|A, \eR(I)) = 0$.
  \end{enumerate}
\end{enumerate}
\end{theorem}
We note the modules  $D_*(\R(I)|A, \R(I))$ and $D_*(\eR(I)|A, \eR(I))$ are graded. So we may consider vanishing of these modules for high degrees. If $L$ is an ideal in a ring $T$ then by $V(L)$ we denote the primes in $T$ containing $L$.
We show
\begin{theorem}\label{second}
(with hypotheses as in \ref{setup}).  Further assume $A$ is \CM \ and $\height I > 0$.
 The following assertions are equivalent:
\begin{enumerate}[\rm (i)]
  \item $\Ass H_1(J) \subseteq \Ass \R(I) \cup V(\R(I)_+)$.
  \item $\Ass H_1(\wh{J}) \subseteq \Ass \eR(I) \cup V(\eR(I)_+)$.
  \item  $D_2(\R(I)|A, \R(I))_n = 0$ for $n \gg 0$.
  \item  $D_2(\eR(I)|A, \eR(I))_n = 0$ for $n \gg 0$.
\end{enumerate}
\end{theorem}

\begin{remark}
Theorems \ref{first}, \ref{second} are not only interesting in its own regard but it is also an essential ingredient in the proofs of Theorems \ref{proj-ci} and \ref{ex-2-thm}.
\end{remark}

\s  \label{poly-m} Let  $(A,\m)$ be a Noetherian local ring and let $I$ be an $\m$-primary ideal. Let  $\R(I) = \bigoplus_{n \geq 0}I^n$ be the Rees algebra of $I$.  It is easily shown that for all $j \geq 1$ the $A$-module  $D_j(\R(I)|A, E)_n$ has finite length for all $n$ (here $E$ is a finitely generated graded $\R(I)$-module); see \ref{poly-growth}. Thus the function $n \rt \ell(D_j(\R(I)|A, E)_n$ is of polynomial type of degree $\leq d - 1$. Here $d = \dim A$. We assume $d \geq 1$.
It is of some interest to find general bounds on this polynomial. Let $S = A[X_1, \ldots, X_l]$ be a graded polynomial algebra mapping onto $\R(I)$ and let $\epsilon \colon S \rt \R(I)$ be this map. We set the degree of the zero polynomial to be $-1$.
\begin{theorem}\label{poly-m-thm}(with hypotheses as in \ref{poly-m}).  Also assume $A$ is \CM. Let $ 1 \leq i \leq d$. The following are equivalent:
\begin{enumerate}[\rm (i)]
  \item For any finitely generated graded $\R(I)$-module $E$ the function\\
$n \rt \ell(D_2(\R(I)|A, E)_n)$ is of polynomial type of degree $\leq d - i- 1$.
  \item For every prime $P \in \Proj(\R(I))$ with $\height P \leq i$ the map $\epsilon_P \colon S_{\epsilon^{-1}(P)} \rt \R(\F)_P$ is a complete intersection.
\end{enumerate}
Furthermore if any of the above conditions hold then for any finitely generated graded $\R(I)$-module  and $j \geq 2$ the function
$n \rt \ell(D_j(\R(I)|A, E)_n)$ is of polynomial type of degree $\leq d - i- 1$.
\end{theorem}

We need to find a single module which can determine a bound on the degree of the above polynomials. We are able to do this when $(A,\m)$ is regular local. We prove
\begin{corollary}\label{poly-m-corr-reg}(with hypotheses as in \ref{poly-m}). Further assume that $A$ is regular local. Let $J = \ker \epsilon$.
Let $ 1 \leq i \leq d$. The following are equivalent:
\begin{enumerate}[\rm (i)]
  \item For any finitely generated graded $\R(I)$-module $E$ the function\\
$n \rt \ell(D_2(\R(I)|A, E)_n)$ is of polynomial type of degree $\leq d - i- 1$.
  \item For every prime $P \in \Proj(\R(I))$ with $\height P \leq i$ the module $H_1(J)_P$ is free.
  \item $\R(I)_P$ is a complete intersection for every prime $P \in \Proj(\R(I))$ with \\ $\height P \leq i$.
  \end{enumerate}
\end{corollary}

We now describe in brief the contents of this paper. In section two we discuss some preliminaries  that we need. In section three we prove some basic results in Andr\"e-Quillen homology of Rees algebras. In section four we prove Theorems \ref{first} and \ref{second}. In section five we prove Theorem \ref{ex-2-thm}. In section six  we prove Theorem \ref{proj-ci}. In the next section we give a proof of Theorem \ref{rank}. In section eight we give a proof of Proposition \ref{locus}. In the next section we give aproof of Theorem \ref{min-gen}. In section ten we give proofs of Theorem \ref{poly-m-thm} and Corollary \ref{poly-m-corr-reg}. Finally in the appendix we give a proof of some results which we believe is already known. However we do not have a reference.

\section{preliminaries}
In this section we discuss some preliminary facts that we need.

\s Let $\psi \colon R \rt S$ be a homomorphism of Noetherian rings. The   Andr\'e-Quillen homology $D_n(S | R, N)$ of the $R$-algebra $S$ with coefficients in an $S$-module $N$ is the $n^{th}$ homology module of $L(S|R)\otimes_S N$, where $L(S|R)$ is the cotangent complex of $\psi$, uniquely defined in the derived category of $S$-modules $D(S)$, see \cite{A} and \cite{Q}.
We follow the exposition in \cite{I}. For $n = 0, 1,2$ a nice treatment has been given in \cite{MR}. If $R$ and $S$ are ($\Z$)-graded and $\psi$ is degree preserving then $L(S|R)$ is a complex of graded $S$-modules and if $N$ is a graded $S$-module then $D_n(S|R, N)$ is a graded $S$-module for all $n \geq 0$.

\s If $\psi \colon R \rt S$ is essentially of finite type then $L(S|R)$ is homotopic to a complex
$$ \cdots  \rt L_n \rt L_{n-1} \rt \cdots \rt L_1 \rt L_0 \rt 0,$$
where each $L_i$ is a finitely generated free $S$-module, see \cite[6.11]{I}.
So if $N$ is a finitely generated $S$-module then so is $D_n(S|R, N)$ for all $n \geq 0$.

\s \label{ff} Let $(A, \m)$ be a Noetherian local ring. Let $T = \bigoplus_{n \in \Z}T_n$ be a $\Z$-graded ring with $T_0 = A$. Assume $\n = (\bigoplus_{n \leq -1}T_n )\oplus \m \oplus (\bigoplus_{n \geq 1}T_n)$ is a maximal ideal of $T$.  Let $I$ be a graded ideal of $T$. If  $T$ is \CM \ then all non-zero Koszul homology modules of $I$ have Krull dimension equal to $\dim T/I$. This is proved for local rings in \cite[4.2.2]{V}. The same proof works for *-local rings.

\s \label{ss} (with hypotheses as in \ref{ff}).  Assume $K = (u_1, \ldots, u_s)$ where $u_i$ are homogeneous. Let $M$ be a graded $T/K$-module. (Here $T$ need not be \CM). Then there exists an exact sequence
\begin{align*}
 0 \rt D_2(T/K|T, M) &\rt H_1(u_1, \ldots, u_s)\otimes_{T/K}M \\
 &\rt (T/K)^s\otimes_{T/K}M \rt K/K^2\otimes_{T/K}M \rt 0.
\end{align*}
This proof is proved for local rings in \cite[2.5.1]{MR}. The same proof works for *-local rings.

\section{Some preliminary results }
In this section we first prove:
\begin{lemma}\label{l}
Let $(A,\m)$ be a Noetherian local ring and let $I \subseteq \m$ be an ideal of $A$. Let $M$ be a finitely generated $A$-module with an $I$-stable filtration $\F = \{ \F_n \}_{n \in \Z}$. Let $\eR(\F, M) = \bigoplus_{n \in \Z}\F_n$ be the extended Rees module of $M$ \wrt \ $\F$. Then for $j \geq 1$ there exists positive integers $s_j$ (depending on $j$) such that
$$t^{-s_j}D_j(\eR(I)|A, \eR(\F, M)) = 0.$$
In particular $D_j(\eR(I)|A, \eR(\F, M))_n = 0$ for $n \ll 0$.
\end{lemma}
\begin{proof} Fix $j \geq 1$. We note that $D_j(\eR(I)|A, \eR(\F, M))$ is a finitely generated graded $\eR(I)$-module. So it suffices to show $D_j(\eR(I)|A, \eR(\F, M))_{t^{-1}} = 0$.
We have an exact sequence of $\eR(I)$-modules
\[
0 \rt \eR(\F, M) \rt M[t, t^{-1}] \rt L_\F(M) = \bigoplus_{n \in \Z}M/\F_n  \rt 0.
\]
We note that $L_\F(M)_n = 0$ for $n \ll 0$. In particular $L_\F(M)_{t^{-1}} = 0$. So \\ $\eR(\F, M)_{t^{-1}} \cong  M[t, t^{-1}]$.
We also have $\eR(I)_{t^{-1}} = A[t, t^{-1}]$ a smooth $A$-algebra.
We have for $j \geq 1$
\begin{align*}
   D_j(\eR(I)|A, \eR(\F, M))_{t^{-1}} &\cong D_j(\eR(I)_{t^{-1}}|A, \eR(\F, M)_{t^{-1}}) \\
  &=  D_j(A[t,t^{-1}]|A, M[t,t^{-1}]) = 0.
\end{align*}
The vanishing in the last line occurs since $A[t,t^{-1}]$ is a smooth $A$-algebra.
\end{proof}

\s Let $M$ be a finitely generated $A$-module. If $\F = \{ \F_n \}_{n \in \Z}$ is an $I$-stable filtration then set $G_\F(M) = \bigoplus_{n \in \Z}\F_n/\F_{n+1}$ the associated graded module of $M$ \ \wrt \ $\F$. We note $G_\F(M)_n = 0$ for $n \ll 0$. We have the following:
\begin{corollary}\label{g-l}
(with hypotheses as above). If for some $j \geq 2$ we have \\ $ D_j(\eR(I)|A, G_\F(M)) = 0$ then $D_{j-1}(\eR(I)|A, \eR(\F, M)) = 0$.
\end{corollary}
\begin{proof}
We have an exact sequence of graded $\eR(I)$-modules
$$ 0 \rt \eR(\F, M)(+1) \xrightarrow{t^{-1}} \eR(\F, M) \rt G_\F(M) \rt 0.$$
By considering the long exact sequence in homology we obtain an injective map
\[
0 \rt D_{j-1}(\eR(I)|A, \eR(\F, M))(+1) \xrightarrow{t^{-1}} D_{j-1}(\eR(I)|A, \eR(\F, M)).
\]
By \ref{l} we know that $D_{j-1}(\eR(I)|A, \eR(\F, M))$ is $t^{-1}$-torsion. The result follows.
\end{proof}

Next we prove:
\begin{lemma}\label{lemm}
Let $(A,\m)$ be a local Noetherian ring and let $I \subseteq \m$ be an ideal.
Let $M$ be an $A$-module and let $\F = \{ \F_n \}_{n \in \Z}$ be an $I$-stable filtration on $M$ with $\F_i = M$ for $i \leq 0$ and $\F_1 \neq F_0$.
Let $\eR(\F, M) = \bigoplus_{n \in Z}\F_n$ be the extended Rees module of $M$ \wrt \ $\F$ (considered as an $\eR(I)$-module).
Let $\R_\F(M) = \bigoplus_{n \geq 0}\F_n$ be the  Rees module of $M$ \wrt \ $\F$ (considered as an $\R(I)$-module).
For each $j \geq 0$ there exists $n(j)$ (depending on $j$) such that we have an isomorphism of $A$-modules
$$D_j(\R(I)|A, \R_\F(M))_n \cong D_j(\eR(I)|A, \eR(\F, M))_n  \quad \text{for} \ n \geq n(j).$$
\end{lemma}
\begin{proof}
We have an exact sequence of $\R(I)$-modules
\[
0 \rt \R_\F(M) \rt \eR(\F, M) \rt E \rt 0.
\]
We note that the cotangent complex of $\R(I)$ \ \wrt \ $A$ is homotopic to a complex of finitely generated graded free $\R(I)$-modules. As $E_n  = 0 $ for $n \gg 0$ it follows that given $j \geq 0$ there exists $n(j)^\prime$ depending on $j$ such that
$D_j(\R(I)|A, E)_n = 0$ for all $n \geq n(j)^\prime$. Set $n(j)^* = \max \{ n(j)^\prime, n(j+1)^\prime \}$.
Therefore for all $j \geq 0$ we have an isomorphism of $A$-modules
$$D_j(\R(I)|A, \R_\F(M))_n \cong D_j(\R(I)|A, \eR(\F, M))_n \quad \text{for all $n \geq n(j)^*$}. $$
Consider the Jacobi-Zariski sequence to $A \rt \R(I) \rt \eR(I)$ for $j \geq 0$
\[
D_j(\R(I)|A, \eR(\F, M)) \rt D_j(\eR(I)|A, \eR(\F, M)) \rt D_j(\eR(I)|\R(I), \eR(\F, M)).
\]
Let $I = (a_1, \cdots, a_s)$.  Set $X_i = a_it$. Note $\R(I)_{X_i} = \eR(I)_{X_i}$. We then have for $j \geq 1$
\[
D_j(\eR(I)|\R(I), \eR(\F, M))_{X_i} \cong D_j(\eR(I)_{X_i}|\R(I)_{X_i}, \eR(\F, M)_{X_i}) = 0.
\]
The last equality holds as $\R(I)_{X_i} = \eR(I)_{X_i}$.
We note that \\ $D_j(\eR(I)|\R(I), \eR(\F, M))$ is a finitely generated graded $\eR(I)$-module. By the above calculation we have that $D_j(\eR(I)|\R(I), \eR(\F, M))$ is $\eR(I)_+$-torsion. Therefore
$D_j(\eR(I)|\R(I), \eR(\F, M))_n = 0$ for $n \gg 0$. The result follows.
\end{proof}
\section{Proofs of Theorems \ref{first} and \ref{second}}
In this section we give proofs of Theorem \ref{first} and \ref{second}.
We first state a few preliminary results first.

\s \label{base-change} We need to change base to the case when $A$ has infinite residue field. If the residue field of $A$ is finite then  let $B = A[Y]_{\m A[Y]}$ with maximal ideal $\n = \m B$. We note that the residue
field of $B$ is $k(Y)$ which is infinite. Let $I = (a_1, \ldots, a_r)$ be an ideal in $A$. Set $T = B[IBt] = \R(I)\otimes_A B$ and $\wh{T} = B[IB, t^{-1}] = \eR(I)\otimes_A B$.

Consider the map $\epsilon \colon S = A[X_1, \ldots, X_r]  \rt \R(I)$ which maps $X_i \rt a_it$ for all $i$. Let $J = \ker \epsilon$.
Consider $\epsilon \otimes_A B \colon  U \rt T$ where $U = B[X_1, \ldots, X_r]  = S\otimes_A B$. We note that $J_B := \ker \epsilon\otimes B = J\otimes_A B$.

Consider the map $\wh{\epsilon} \colon \wh{S} = A[X_1, \ldots, X_r, V]  \rt \eR(I)$ which maps $X_i \rt a_it$ for all $i$ and $V$ is mapped to $t^{-1}$. Let $\wh{J} = \ker \wh{\epsilon}$.
Consider $\wh{\epsilon}\otimes_A B \colon  \wh{U} \rt \wh{T}$ where $\wh{U} = B[X_1, \ldots, X_r, V]  = \wh{S}\otimes_A B$. We note that $\wh{J_B }:= \ker \wh{\epsilon}\otimes B = \wh{J}\otimes_A B$.

\begin{lemma}\label{inf-bc}
(with hypotheses as in \ref{base-change}).  Further assume $A$ is \CM \ and $\height I > 0$.
 We have
\begin{enumerate}[\rm (I)]
  \item The following assertions are equivalent:
  \begin{enumerate}[\rm (a)]
    \item $H_1(J)$ is an unmixed and equi-dimensional $\R(I)$-module.
    \item $H_1(J_B)$ is an unmixed and equi-dimensional $T$-module.
    \end{enumerate}
    \item The following assertions are equivalent:
  \begin{enumerate}[\rm (a)]
  \item $H_1(\wh{J})$ is an unmixed and equi-dimensional $\eR(I)$-module.
  \item $H_1(\wh{J_B})$ is an unmixed and equi-dimensional $\wh{T}$-module.
  \end{enumerate}
  \item Fix $n \in \Z$. The following assertions are equivalent:
  \begin{enumerate}[\rm (a)]
  \item $D_2(\R(I)|A, \R(I))_n = 0$
  \item $D_2(T|B, T)_n = 0$.
  \end{enumerate}
  \item Fix $n \in \Z$. The following assertions are equivalent:
  \begin{enumerate}[\rm (a)]
  \item $D_2(\eR(I)|A, \eR(I))_n = 0$
  \item $D_2(\wh{T}|B, \wh{T})_n = 0$.
  \end{enumerate}
   \item The following assertions are equivalent:
   \begin{enumerate}[\rm (a)]
     \item $\Ass H_1(J) \subseteq \Ass \R(I) \cup V(\R(I)_+)$.
     \item $\Ass H_1(J_B) \subseteq \Ass T \cup V(T_+)$.
   \end{enumerate}
 \item The following assertions are equivalent:
   \begin{enumerate}[\rm (a)]
     \item $\Ass H_1(\wh{J}) \subseteq \Ass \eR(I) \cup V(\eR(I)_+)$.
     \item $\Ass H_1(\wh{J_B}) \subseteq \Ass \wh{T} \cup V(\wh{T}_+)$.
   \end{enumerate}
\end{enumerate}
\end{lemma}

To prove this result we need Theorem 23.3 from \cite{Mat}.   Unfortunately  there is a typographical error in the statement of Theorem 23.3 in \cite{Mat}.
So we state it here.
\begin{theorem}\label{23}
Let $\varphi \colon A \rt B$ be a homomorphism of Noetherian rings, and let $E$ be an $A$-module and $G$ a $B$-module. Suppose that
$G$ is flat over $A$; then we have the following:
\begin{enumerate}[\rm (i)]
\item
if $\fp \in \Spec A$ and $G/\fp G \neq 0$ then
\[
 ^a \varphi \left( \Ass_B(G/\fp G)  \right) = \Ass_A (G/\fp G) = \{\fp \}.
\]
\item
$\displaystyle{\Ass_B(E\otimes_A G) = \bigcup_{\fp \in \Ass_A(E) } \Ass_B(G/\fp G).}$
\end{enumerate}
\end{theorem}
\begin{remark}
In  \cite{Mat} $\Ass_A(E\otimes G)$ is typed instead of $\Ass_B(E\otimes G)$. Also note that
$^a \varphi(\fP) = \fP \cap A$ for $\fP \in \Spec B$.
\end{remark}
We now give
\begin{proof}[Proof of Lemma \ref{inf-bc}]
(I) We note that $T$ is faithfully flat over $R$ and $H_1(J_B) = H_1(J)\otimes_R T$. To prove the result we may assume $H_1(J) \neq 0$.
By Theorem \ref{23} we have
$$ \Ass_T H_1(J_B) = \bigcup_{\fp \in \Ass_{\R(I)} H_1(J) } \Ass_T T/\fp T. $$
We also have
$$ \Ass_T T = \bigcup_{\fp \in \Ass_{\R(I)} \R(I) } \Ass_T T/\fp T. $$
We also note that $\dim H_1(J) = \dim \R(I)$ and $\dim H_1(J_B) = \dim T$; see \ref{ff}. We also have  $\dim \R(I) = \dim T$.

(a) $\implies$ (b) We note that as $H_1(J)$ is
an unmixed and equi-dimensional $\R(I)$-module and $\dim H_1(J) = \dim \R(I)$; it follows that $\Ass H_1(J) \subseteq \Ass \R(I)$. By the above equality we have
$\Ass_T H_1(J_B) \subseteq \Ass_T T$. As $T$ is unmixed and equi-dimensional; see \ref{unmixed}; the result follows.

(b) $\implies$ (a)
 We note that as $H_1(J_B)$ is
an unmixed and equi-dimensional $T$-module and $\dim H_1(J_B) = \dim T$; it follows that $\Ass H_1(J_B) \subseteq \Ass T$.
Let $\fp \in \Ass H_1(J)$. As $T$ is faithfully flat over $\R(I)$ we get that $T/\fp T \neq 0$. Let $Q \in \Ass_T  T/\fp T$. Then note that as $Q \in \Ass T$ we get that $\fp = Q\cap \R(I) \in \Ass \R(I)$. So $\Ass H_1(J) \subseteq \Ass \R(I)$.  As $\R(I)$ is unmixed and equi-dimensional; see \ref{unmixed}; the result follows.

(II) This follows as in (I).

(III) By \cite[6.3]{I} we have for all $j \geq 0$ we have a graded isomorphism
$$D_j(T|B, T) = D_j(\R(I)|A, T)$$
We note that $D_j(\R(I)|A, T) = D_j(\R(I)|A, \R(I))\otimes B$. So  for all $n \in \Z$
$$D_j(T|B, T)_n = D_j(\R(I)|A, \R(I))_n\otimes B. $$
The result follows as $B$ is faithfully flat over $A$.

(IV) This follows as in (III).

(V) We note that if $\fp \in \Spec \R(I)$ and $Q\in \Ass_T T/\fp T$ then
\begin{enumerate}
  \item $Q\cap \R(I) = \fp.$.
  \item $\fp \supseteq \R(I)_+$ if and only if $Q \supseteq T_+$.
  \item $\fp \in \Ass \R(I)$ if and only if $Q \in \Ass T$.
\end{enumerate}
(a) $\implies$ (b) Let $Q \in \Ass_T H_1(J_B)$. Then $Q \in \Ass T/\fp T$ for some \\ $\fp \in \Ass_{\R(I)} H_1(J)$. If $\fp \supseteq \R(I)_+$ then $Q \supseteq T_+$. If $\fp \in \Ass \R(I)$ then $Q \in \Ass T$.

(b) $\implies$ (2). Let $\fp \in \Ass H_1(J)$. The map $\R(I) \rt T$ is faithfully flat. So $T/\fp T \neq 0$. Let $Q \in \Ass_T T/\fp T$. Then $Q \in \Ass_T H_1(J_B)$ and $Q \cap \R(I) = \fp$.
If $Q \supseteq T_+$ then $\fp \supseteq \R(I)_+$. If $Q \in \Ass T$ then $\fp \in \Ass \R(I)$. The result follows.

(VI) This is as in (V).
\end{proof}

We give
\begin{proof}[Proof of Theorem \ref{first}]
.

(II) We note that $D_2(\eR(I)|A, \eR(I)) = D_2(\eR(I)|\wh{S}, \eR(I))$. We also have $\eR(I)$ is unmixed and equi-dimensional (see \ref{unmixed}).
By \ref{ss} we have an exact sequence
\[
0 \rt D_2(\eR(I)|\wh{S}, \eR(I)) \rt H_1(\wh{J}) \rt \eR(I)^s \cdots.
\]

(II) (b) $\implies$ (a) follows as $\eR(I)$ is unmixed and equi-dimensional, see \ref{unmixed}.

(a) $\implies$ (b). Suppose if possible $D_2(\eR(I)|\wh{S}, \eR(I)) \neq 0$. By Lemma \ref{l} we have $D_2(\eR(I)|\wh{S}, \eR(I))$ is $t^{-1}$-torsion. So $\dim D_2(\eR(I)|\wh{S}, \eR(I)) \leq d$.
But $H_1(I)$ is unmixed and equi-dimensional; and has dimension $= \dim \eR(I) = d +1$. This is a contradiction.

(I) We note that $D_2(\R(I)|A, \R(I)) = D_2(\eR(I)|{S}, \R(I))$. We also have $\eR(I)$ is unmixed and equi-dimensional (see \ref{unmixed}).
By \ref{ss} we have an exact sequence
\[
0 \rt D_2(\R(I)|S, \R(I)) \rt H_1(J) \rt \R(I)^l \cdots.
\]

(I) (b) $\implies$ (a) follows as $\R(I)$ is unmixed and equi-dimensional.

(a) $\implies$ (b). By \ref{inf-bc} we may assume that the residue field $k$ of $A$ is infinite. Let $I = (a_1, \ldots, a_s)$. By \ref{generic}, we may assume that each $a_i$ is $A$-regular. Set $X_i = a_it$ for $i = 1, \ldots, s$. Suppose if possible $E = D_2(\R(I)|{S}, \R(I)) \neq 0$.

 Claim-1 $E$ is $\R(I)_+$-torsion.

 Suppose if possible this is not true.
then $E_{X_i} \neq 0$ for some $i$. But
\begin{align*}
   E_{X_i} = D_2(\R(I)|A, \R(I))_{X_i}& = D_2(\R(I)_{X_i}|A, \R(I)_{X_i}) \\
  &= D_2(\eR(I)_{X_i}|A, \eR(I)_{X_i}) \\
  &=  D_2(\eR(I)|A, \eR(I))_{X_i}.
\end{align*}
So $E_{X_i}$ is $t^{-1}$-torsion (by \ref{l}).
By hypothesis $H_1(J)$ is unmixed, equi-dimensional  graded $\R(I)$-module of dimension $\dim \R(I)$.
We have $\Ass H_1(J)_{X_i} \subseteq \Ass \R(I)_{X_i}$. But $\R(I)_{X_i} = \eR(I)_{X_i}$.
We have $\Ass \R(I)_{X_i} = \{ Q_{X_i} \mid Q \in \Ass \eR(I) \}$.

Claim-2:  $t^{-1} \notin  Q_{X_i}$ for every $ Q \in \Ass \eR(I)$.

Suppose Claim-2 is true. Then $t^{-1}$ is $H_1(J)_{X_i}$-regular. But  $E_{X_i}$ is $t^{-1}$-torsion and a submodule of $H_1(J)_{X_i}$. This is a contradiction. So Claim-1 follows.

We prove Claim-2. Suppose  $t^{-1} \in  Q_{X_i}$ for some $ Q \in \Ass \eR(I)$. As $\eR(I)$ is unmixed, by \cite[4.5.5]{BH} we have $Q = PA[t, t^{-1}]\cap \eR(I)$ for some minimal prime $P$ of $A$.
We have $X_i^l t^{-1} \in Q$ for some $l \geq 0$. Note that $l = 0$ is not possible. Also note that if $l \geq 1$ then $a_i \in P$. This is a contradiction as $a_i$ is $A$-regular and $P \in \Ass A$.

So $E_{X_i} = 0$ for all $i$.
Thus $E$ is a finitely generated $\R(I)/\R(I)_+^m$  module for some $m \geq 1$. So $E$ is a finitely generated $A$-module. Therefore $\dim E \leq d$. But $\dim H_1(J)$ is $d + 1$ and it is unmixed and equi-dimensional. This is a contradiction. So $E = 0$.
\end{proof}

Next we give
\begin{proof}[Proof of Theorem \ref{second}]
The assertion (iii) $ \Leftrightarrow$ (iv) follows from Lemma \ref{lemm}.
For the rest of the assertions  we first note that $D_2(\eR(I)|A, \eR(I)) = D_2(\eR(I)|\wh{S}, \eR(I))$ and  $D_2(\R(I)|A, \R(I)) = D_2(\eR(I)|{S}, \R(I))$.

By \ref{ss} we have an exact sequence
\[
0 \rt D_2(\eR(I)|\wh{S}, \eR(I)) \rt H_1(\wh{J}) \rt \eR(I)^s \cdots.
\]
(iv) $\implies$ (ii):
 We note that $D_2(\eR(I)|{S}, \R(I))$ is a finitely generated graded $\R(I)$-module. As  $D_2(\eR(I)|{S}, \R(I))_n = 0$ for $n \gg 0$ it follows that $D_2(\eR(I)|{S}, \R(I))$
is $\R(I)_+$-torsion. The result follows from the above exact sequence.

(iii) $\implies$ (i): This is similar to (iv) $\implies$ (ii).

(i)  $\implies$ (iii) and  (ii)  $\implies$ (iv):
By \ref{inf-bc} we may assume that the residue field $k$ of $A$ is infinite. Let $I = (a_1, \ldots, a_s)$. By \ref{generic} we may assume that each $a_i$ is $A$-regular.
 Set $X_i = a_it$ for $i = 1, \ldots, s$. We also note that
\begin{align*}
  E =   D_2(\R(I)|A, \R(I))_{X_i}&\cong  D_2(\R(I)_{X_i}|A, \R(I)_{X_i}) \\
  &= D_2(\eR(I)_{X_i}|A, \eR(I)_{X_i}) \\
  &=  D_2(\eR(I)|A, \eR(I))_{X_i}.
\end{align*}
Suppose if possible $E \neq 0$. By \ref{l} we get that $E$ is $t^{-1}$-torsion.  By our hypotheses $\Ass H_1(J)_{X_i} \subseteq \Ass \R(I)_{X_i}$ and $\Ass H_1(\wh{J})_{X_i} \subseteq \Ass \eR(I)_{X_i}$.
 We note that $\R(I)_{X_i} = \eR(I)_{X_i}$.
By an argument similar to that  of Claim-2 in  proof of Theorem \ref{first},  we get that $t^{-1}$ is $H_1(J)_{X_i}$ and $H_1(\wh{J})_{X_i}$-regular. Also $t^{-1}$ is $E$-torsion. This is a contradiction. So $E$ is zero. It follows that $D_2(\R(I)|A, \R(I))$  (and $D_2(\eR(I)|A, \eR(I))$) is supported at $V(\R(I)_+)$ (respectively $V(\eR(I)_+)$. The result follows.
\end{proof}
\section{Proof of Theorem \ref{ex-2-thm}}
In this section we first show that while proving \ref{ex-2-thm} we can assume that $A$ is complete.
\begin{lemma}\label{comp-2}
(with hypotheses as in \ref{ex-2-thm}) Let $A^*$ be the completion of $A$. Set $\eR = \eR(I)$ and  $T = A^*[IA^*t, t^{-1}] = \eR \otimes_A A^*$.   The natural map $\wh{S}\otimes_A A^* \rt T$ has kernel $\wh{J}^* = \wh{J} \otimes_A A^*$. We note that $H_1(\wh{J}^*) = H_1(\wh{J})\otimes_{\eR} T$. We have
\begin{enumerate}[\rm (I)]
  \item The following assertions are equivalent:
  \begin{enumerate}[\rm (a)]
    \item $\eR$ is a complete intersection.
    \item $T$ is a complete intersection.
  \end{enumerate}
  \item The following assertions are equivalent:
  \begin{enumerate}[\rm (a)]
    \item $H_1(\wh{J})$ is a free $\eR$-module.
    \item $H_1(\wh{J}^*)$ is a free $T$-module.
  \end{enumerate}
  \item The following assertions are equivalent:
  \begin{enumerate}[\rm (a)]
  \item $D_3(\eR|A, \eR) = 0$.
  \item $D_3(T|A^*, T) = 0$.
  \end{enumerate}
\end{enumerate}
\end{lemma}
\begin{proof}
(I) The map $\eR \rt T$ is flat by fiber $k$. The result follows from \cite[4.3.8]{MR}.

(II) We have $\Tor^T_1(k, H_1(\wh{J}^*) ) \cong \Tor^{\eR}_1(k, H_1(J)) \otimes_{\eR} T$. The result follows as $T$ is a faithfully flat $\eR$-algebra.

(III) As $A^*$ is a flat $A$-algebra we have
$D_3(T|A^*, T) =  D_3(\eR|A, T)$, see \cite[6.3]{I}. Also note that $D_3(\eR|A, T) = D_3(\eR|A, \eR)\otimes_A A^*$. The result follows as $A^*$ is a faithfully flat $A$-algebra.
\end{proof}
In this section we first give
\begin{proof}[Proof of Theorem \ref{ex-2-thm}]:
By \ref{comp-2} we may assume that $A$ is complete. Let $Q$ be a complete regular local ring with a surjection $ \psi \colon Q \rt A$. We have $\ker \psi $ is generated by a $Q$-regular sequence $\fb = f_1, \ldots, f_c$.
Let $k = A/\m$. Set $\eR = \eR(I)$.

(i) $\implies$ (ii): \\
Consider $A \rt \eR \rt k$. The Jacobi-Zariski sequence yields for $j \geq 3$
$$D_{j+1}(k|\eR,  k) \rt D_j(\eR|A, k) \rt D_j(k|A,  k). $$
As $A$ is a complete intersection we have  $ D_j(k|A,  k) = 0$ for $j \geq 3$. Also as $\eR$ is a complete intersection we have $D_{j}(k|\eR,  k) = 0$ for $j \geq 3$. Therefore
$D_j(\eR|A, k) = 0$ for $j \geq 3$. From \cite[8.7]{I} it follows that $D_j(\eR|A, E) = 0$ for $j \geq 3$ for any $\eR$-module $E$. In particular this holds when $E = G_\F(M)$ the associated graded module of
an $I$-stable  filtration on an $A$-module $M$.   By \ref{g-l} it follows that $D_j(\eR|A, \eR(\F, M)) = 0$ for $j \geq 2$.

Recall we have a right exact complex  $C \colon 0 \rt H_1(\wh{J}) \xrightarrow{\theta} \eR^s \rt \wh{J}/\wh{J^2} \rt 0$ where $\ker \theta = D_2(\eR|A, \eR)$. The later module is zero. So $C$ is exact. It follows that $\Tor^{\eR}_1(\wh{J}/\wh{J}^2, \eR(\F, M)) = D_2(\eR|A, \eR(\F, M)) = 0$.
Consider the exact sequence $ 0 \rt \M \rt \eR \rt k \rt 0$. As $t^{-1}$ is $\M$-regular we get that $\M = \eR(\F, N)$ for some $A$-module $N$ and a $I$-stable filtration on $N$, see \cite[3.1]{PTJ-2}. It follows that $\Tor^{\eR}_2(\wh{J}/\wh{J}^2, k) = 0$.
So $\projdim_{\eR} \wh{J}/\wh{J}^2 \leq 1$.
As $C$ is exact it follows that $H_1(\wh{J})$ is free $\eR$-module.

(ii) $\implies$ (i): \\
As $H_1(J)$ is free $\eR$-module it is unmixed and equ-dimensional (as $\eR$ is). So by Theorem \ref{first} (II) we get \\ $D_2(\eR|A, \eR) = 0$. So the complex  $C \colon 0 \rt H_1(\wh{J}) \xrightarrow{\theta} \eR^s \rt \wh{J}/\wh{J^2} \rt 0$ is exact. This implies that $\projdim_{\eR} \wh{J}/\wh{J}^2 \leq 1$. So we have
\[
D_2(\eR|A, \M) = \Tor^{\eR}_1(\M,\wh{J}/\wh{J}^2) = \Tor^{\eR}_2(k,\wh{J}/\wh{J}^2)  = 0.
\]
We consider the exact sequence $0 \rt \M \rt \eR \rt k \rt 0$. As $D_3(\eR|A,\eR) = 0$ and $D_2(\eR|A, \M) = 0$ it follows from the long exact sequence of homology that $D_3(\eR|A, k) = 0$.
By the Jacobi-Zariski sequence for $Q \rt A \rt \eR$ it follows that $D_3(\eR|Q, k) = 0$. It follows that $D_3(\eR|Q, -) = 0$. Now $\eR$ has finite flat dimension over $Q$. By \cite[1.4]{Av}, it follows that $\eR$ is a complete intersection.
\end{proof}

\section{Proof of Theorem \ref{proj-ci}}
\s \label{proj-ci-2} We first prove that it suffices to assume that $A$ is a quotient of a regular local ring. Let $(A,\m)$ be an excellent complete intersection and let $I$ be an ideal in $A$ of positive height. Let $A^*$ be the completion of $A$. Let $T = A^*[IA^*t] = \R(I)\otimes_A A^*$ and let $\wh{T} = A^*[IA^*t, t^{-1}] = \eR(I)\otimes_A A^*$. The natural map $S\otimes_A A^* \rt T$ has kernel $J^* = J\otimes_A A^*$. We note that $H_1(J^*) = H_1(J)\otimes_{R(I)} T$. The natural map $\wh{S}\otimes_A A^* \rt \wh{T}$ has kernel $\wh{J}^* = \wh{J} \otimes_A A^*$. We note that $H_1(\wh{J}^*) = H_1(\wh{J})\otimes_{\eR(I)} \wh{T}$.

\begin{lemma}\label{proj-ci-lemm}(with hypotheses as in \ref{proj-ci-2} We have
\begin{enumerate}[\rm (I)]
  \item The following assertions are equivalent:
   \begin{enumerate}[\rm (a)]
     \item $\Proj \R(I)$ is a complete intersection.
     \item $\Proj T$ is a complete intersection.
   \end{enumerate}
  \item The following assertions are equivalent:
   \begin{enumerate}[\rm (a)]
     \item $\Proj \eR(I)$ is a complete intersection.
     \item $\Proj \wh{T}$ is a complete intersection.
   \end{enumerate}
  \item The following assertions are equivalent:
   \begin{enumerate}[\rm (a)]
   \item $D_3(\R(I)|A, \R(I))_n = 0$ for $n \gg 0$.
   \item $D_3(T|A^*, T)_n = 0$ for $n \gg 0$.
   \end{enumerate}
   \item The following assertions are equivalent:
   \begin{enumerate}[\rm (a)]
   \item $D_3(\eR(I)|A, \eR(I))_n = 0$ for $n \gg 0$.
   \item $D_3(\wh{T}|A^*, \wh{T})_n = 0$ for $n \gg 0$.
   \end{enumerate}
    \item The following assertions are equivalent:
   \begin{enumerate}[\rm (a)]
   \item $H_1(J)_P$ is a free $\R(I)_P$-module for  every $P \in \Proj(\R(I))$.
   \item $H_1(J^*)_P$ is a free $T_Q$-module for  every $Q \in \Proj(T)$.
   \end{enumerate}
   \item The following assertions are equivalent:
   \begin{enumerate}[\rm (a)]
   \item $H_1(\wh{J})_P$ is a free $\eR(I)_P$-module for  every $P \in \Proj(\eR(I))$.
   \item $H_1(\wh{J}^*)_P$ is a free $\wh{T}_Q$-module for  every $Q \in \Proj(\wh{T})$.
   \end{enumerate}
\end{enumerate}
\end{lemma}
\begin{proof}
  Let $P$ be a prime ideal in $\R(I)$ and $Q$ a prime ideal in $T$ with $Q\cap \R(I) = P$. Then $P \supseteq \R(I)_+$ if and only if $Q \supseteq T_+$.
  The map $A \rt A^*$ is regular. As $\R(I)$ is a finitely generated $A$-algebra it follows that the map $\R(I) \rt \R(I)\otimes_A A^* = T$ is also regular (and faithfully flat).

  (I) First assume that $\Proj \R(I)$ is a complete intersection. If  $Q \in \Proj(T)$ then note that $P = \R(I)\cap Q$  is in $\Proj(\R(I))$. The map $\R(I)_P \rt T_Q$ is regular. In particular its fiber is a regular local ring. So by a result of Avramov it follows that $T_Q$ is also a complete intersection, see \cite[4.3.8]{MR}.

  Conversely assume that $\Proj T$ is a complete intersection. Let $P \in \Proj \R(I)$. As the map $\R(I) \rt T$ is faithfully flat it follows that there exists prime ideal $Q$ in $T$ with $Q\cap \R(I) =P$. If $Q$ is not homogeneous then if $Q^*$ is the prime ideal in $T$ generated by homogeneous elements in $Q$ we get that $Q^* \supseteq P$. Thus $Q^*\cap \R(I) = P$. Note necessarily $Q^* \in \Proj T$. Thus we can assume there exists $Q\in \Proj T$ such that $Q \cap \R(I) = P$. The map $\R(I)_P \rt T_Q$ is flat. So by a result of Avramov it follows that $\R(I)_P$ is also a complete intersection, see  \cite[4.3.8]{MR}. The result follows.

  (II) This follows from a similar argument as in (I).

  (III) As $A^*$ is a flat $A$-algebra we have
$D_3(T|A^*, T) =  D_3(\R(I)|A, T)$, see \cite[6.3]{I}. Also note that $D_3(R(I)|A, T) = D_3(\R(I)|A, \eR)\otimes_A A^*$. This induces an isomorphism
$D_3(T|A^*, T)_n = D_3(\R(I)|A, \eR)_n\otimes_A A^*$ for all $n \in \Z$.
 The result follows as $A^*$ is a faithfully flat $A$-algebra.

 (IV) This follows from a similar argument as in (III).

 (V) Let $I = (a_1, \ldots, a_m)$. Set $X_i = a_it \in \R(I)_1$ and $Y_i  = a_i t\in T_1$. The map $\R(I)_{X_i} \rt T_{Y_i}$ is faithfully flat for all $i$.
 It suffices to show that if $M$ is a graded $\R(I)$-module then $M$ is a projective $\R(I)_{X_i}$-module if and only if $(M\otimes_{\R(I)}T)_{Y_i}$ is a projective $T_{Y_i}$-module for all $i$.

 If $M_{X_i}$ is projective $\R(I)_{X_i}$-module then clearly $M_{X_i}\otimes_{\R(I)_{X_i}}T_{Y_i}$ is a projective $T_{Y_i}$-module. Observe that
 $$(M\otimes_{\R(I)}T)_{Y_i} = M_{X_i}\otimes_{\R(I)_{X_i}}T_{Y_i}.$$

 Conversely assume $M_{X_i}\otimes_{\R(I)_{X_i}}T_{Y_i}$ is a projective $T_{Y_i}$-module. Let $0 \rt N \rt F \rt M_{X_i} \rt 0$ be an exact sequence where $F$ is a free $\R(I)_{X_i}$-module. So $N$ is also graded. Then notice that
 $$\Ext^1_{\R(I)_{X_i}}(M_{X_i}, N)\otimes_{\R(I)_{X_i}} T_{Y_i} = \Ext^1_{T_{Y_i}}(M_{X_i}\otimes_{\R(I)_{X_i}}T_{Y_i}, N\otimes_{\R(I)_{X_i}}T_{Y_i})  = 0 $$
 As the map $\R(I)_{X_i} \rt T_{Y_i}$ is faithfully flat we get $\Ext^1_{\R(I)_{X_i}}(M_{X_i}, N) = 0$. So the map $F \rt M_{X_i}$ splits. Thus $M_{X_i}$ is a projective $\R(I)_{X_i}$-module.

 (VI) This follows from a similar argument as in (V).
\end{proof}

\begin{proof}[Proof of Theorem \ref{proj-ci}]
By \ref{proj-ci-lemm} we may assume $A$ is a quotient of a regular local ring $(Q, \n)$.
Let $A = Q/(\fb)$ where $\fb = f_1, \ldots, f_c \subseteq \n^2$ is a $Q$-regular sequence.

(i) $\Leftrightarrow$ (ii): This follows as $\Proj \R(I) \cong \Proj \eR(I)$.

(i) $\implies$ (iii):
Let $M$ be a finitely generated  $A$-module and let $\F$ be an $I$-stable filtration on $M$. Set $\R(\F, M)$ be the Rees module  of $M$ \wrt \ $\F$ and let $\eR(\F, M)$ be the extended Rees module associated to $M$ \wrt \ $\F$. Let $E$ be a finitely generated graded  $\R(I)$-module.

Claim-1: For $j \geq 2$ there exists $n(j)$ depending on $j$ and $\R(\F, M)$ such that  $D_j(\R(I)|A, \R(\F, M))_n  = 0$ for $n \geq n(j)$.

 Let $Q'$ be a polynomial algebra over $Q$ mapping onto $\R(I)$.
We note that
$D_j(\R(I)|Q, E) = D_j(\R(I)|Q', E)$ for $j \geq 2$. Let $I = (a_1, \ldots, a_r)$. Set $X_i = a_it$ and let $Y_i$ be an inverse image of $X_i$ in $Q'$. The surjective map
$Q_{Y_i} \rt \R(I)_{X_i}$ is locally a complete intersection as both  $\R(I)_{X_i}$ is a complete intersection and $Q'_{Y_i}$ is a regular ring.
So we have $D_j(\R(I)|Q', E)_{X_i} = 0$ for all $i$. It follows that $D_j(\R(I)|Q', E)$ is $\R(I)_+$-torsion. It follows that
for $j \geq 2$ there exists $n(j)$ depending on $j$ and $E$ such that
 $D_j(\R(I)|Q', E)_n = 0$ for all $n \geq n(j)$.

We consider the Jacobi-Zariski sequence for $Q \rt A \rt \R(I)$. We have \\ $D_j(A|Q,-) = 0$ for $j \geq 2$.
So for $j \geq 3$ we have $D_j(\R(I)|Q, E) \cong D_j(\R(I)|A, E)$. Thus $D_j(\R(I)|A, E)_n = 0$ for $n \gg 0$, for $j \geq 3$.

For $j = 2$ we set $E = \R(\F, M)$. We have an exact sequence
\begin{align*}
  0 = D_2(A|Q, \R(\F, M))
   &\rt D_2(\R(I)|Q, \R(\F, M)) \rt D_2(\R(I)|A, \R(\F, M))  \\
  \rt D_1(A|Q, \R(\F, M))&= (\fb)/(\fb)^2 \otimes_A \R(\F,M) = \R(\F, M)^c.
\end{align*}
We have $D_2(\R(I)|Q, \R(\F, M))_{X_i} = 0$. So
we have $D_2(\R(I)|A, \R(\F, M))_{X_i} $ is a $\R(I)_{X_i}$ submodule of $\R(\F, M)^c_{X_i} = \eR(\F,  M)^c_{X_i}$. Note $\R(I)_{X_i} = \eR(I)_{X_i}$. We also have
\begin{align*}
  D_2(\R(I)|A, \R(\F, M))_{X_i} &\cong
D_2(\R(I)_{X_i}|A, \R(\F, M)_{X_i})  \\
  &\cong  D_2(\eR(I)_{X_i}|A, \eR(\F, M)_{X_i}).
\end{align*}
If $\eR(\F,  M)^c_{X_i} = 0$ then  $D_2(\R(I)|A, \R(\F, M))_{X_i} = 0$. If $\eR(\F,  M)^c_{X_i} \neq 0$  then $t^{-1}$ is $\eR(\F,  M)^c_{X_i}-$ regular. However by \ref{l} we get that
$D_2(\eR(I)_{X_i}|A, \eR(\F, M)_{X_i})$ is $t^{-1}$-torsion. It follows that $D_2(\eR(I)_{X_i}|A, \eR(\F, M)_{X_i}) = 0$.

Thus $D_2(\R(I)|A, \R(\F, M)) $ is supported on $V(\R(I)_+)$. It follows that \\ $D_2(\R(I)|A, \R(\F, M))_n = 0 $ for $n \gg 0$.

Notice Claim-1 proves the assertion (a). We prove (b).

Consider the right exact complex $C \colon 0 \rt H_1(J) \xrightarrow{\theta} \R(I)^s \rt J/J^2 \rt 0 $. Note $\ker \theta = D_2(\R(I)|A, \R(I))$ which vanishes in high degrees. Let $W = \image \theta$. We note that if $\F$ is an $I$-stable filtration on $M$ then $(H_1(J)\otimes \R(\F, M))_n \cong (W\otimes\R(\F, M))_n$ for $n \gg 0$.
It follows that $$\Tor^{\R(I)}_1(J/J^2, \R(\F, M))_n \cong D_2(\R(I)|A, \R(\F, M))_n = 0 \ \text{ for $n \gg 0$.}$$ So  $\Tor^{\R(I)}_1(J/J^2, \R(\F, M))$ is $\R(I)_+$-torsion.

Let $P \in \Proj(\R(I)) = \Proj(\eR(I))$. So $\R(I)_P \cong \eR(I)_P$. Let  $\kappa(P)$ be the residue field of $\eR(I)_P$.
We have an exact sequence $0 \rt E \rt \eR(I) \rt \eR(I)/P \rt 0$. As $t^{-1}$ is $E$-regular, by \cite[3.1]{PTJ-2} we get that $E = \eR(\F, N)$ for some module $N$ and $I$-stable filtration $\F$ on $N$.
We note that $\eR(\F, N)_P = \R(\F, N)_P$ and  $\eR(I)_P = \R(I)_P$. As $\Tor^{\R(I)}_1(J/J^2, \R(\F, N))$ is $\R(I)_+$-torsion we get $\Tor^{\R(I)_P}_1((J/J^2)_P, \R(\F, N)_P) = 0$.
So we have $\Tor^{\R(I)_P}_2((J/J^2)_P, \kappa(P)) = 0$. Thus $\projdim_{\R(I)_P} (J/J^2)_P \leq 1$. We note that as $D_2(\R(I)|A, \R(I))_n = 0 $ for $n \gg 0$ we get $C_P$ is an exact complex. So $H_1(J)_P$ is free $\R(I)_P$-module.

(ii) $\implies$ (iv): This is similar to (i) $\implies$ (iii).

(iv) $\implies$ (ii). Set $\eR = \eR(I)$.  We first note that as $H_1(\wh{J})_P$ is free for all $P \in \Proj(\eR)$ we get $\Ass H_1(\wh{J}) \subseteq V(\eR(I)_+) \cup \Ass \eR$. It follows from Theorem \ref{second} that $D_2(\eR|A, \eR)_n = 0$ for $n \gg 0$. Let $P$ be a prime in $\Proj(\eR)$. Consider the exact sequence
\begin{equation*}
 0 \rt E \rt \eR \rt \eR/P \rt 0   \tag{$\dagger$}
\end{equation*}
Consider the right exact complex $C \colon 0 \rt H_1(\wh{J}) \xrightarrow{\theta} \eR^s \rt \wh{J}/\wh{J}^2 \rt 0 $.
We note that $\ker \theta = D_2(\eR|A, \eR)_n = 0$ for $n \gg 0$. So $\ker \theta_L = 0$ for every $L \in \Proj \R(I)$.
Thus $\projdim (\wh{J}/\wh{J}^2)_L \leq 1$ for every $L \in \Proj(\eR)$. So $\Tor^{\eR_L}_1(E_L, (\wh{J}/\wh{J}^2)_L) = 0$. It follows that
$\Tor^{\eR}_1(E, \wh{J}/\wh{J}^2)$ is $\eR_+$-torsion. So $\Tor^{\eR}_1(E, \wh{J}/\wh{J}^2)_n = 0$ for $n \gg 0$.
 Let $W = \image \theta$. We note that  $(H_1(\wh{J})\otimes E)_n \cong (W\otimes E)_n$ for $n \gg 0$.
It follows that $D_2(\eR|A, E)_n \cong \Tor^{\R(I)}_1(\wh{J}/\wh{J}^2, E)_n  = 0 $ for $n \gg 0$.  So  $D_2(\eR|A, E)$ is $\eR_+$-torsion.
By the exact sequence $(\dagger)$ we get  an exact sequence
$$D_3(\eR|A, \eR) \rt D_3(\eR|A, R/P) \rt D_2(\eR|A, E). $$
Let $\kappa(P)$ be the residue field of $\eR_P$. Then by the above exact sequence we get $D_3(\eR_P|A, \kappa(P)) = 0$. So by \cite[8.7]{I} it follows that $D_3(\eR_P|A, -) = 0$. As $A = Q/(\fb)$ we get
that $D_3(\eR_P|Q, -) = 0$. Now $\eR_P$ has finite flat dimension over $Q$. It follows from \cite[1.4]{Av} that $\eR_P$ is a complete intersection.

(iii) $\implies$ (i): This is similar to (iv) $\implies$ (ii).
\end{proof}

\s\label{ex-1} \emph{Resolution of singularities:}\\
Let $k$ be a field of characteristic zero and let $(A,\m)$ be a reduced local ring essentially of finite type over $k$.
Hironaka, see \cite[Main Theorem, p.\ 132]{Hir}, proved that there exists an ideal $I\subseteq A$ such that:
\begin{enumerate}
  \item $V(I) = \Sing(A)$.
  \item The natural projection morphism $\pi \colon \Proj \R(I) \rt \Spec(A)$ is a resolution of singularities of $\Spec(A)$, i.e.,
$\Proj \R(I)$ is non-singular and $\pi$ induces a $k$-scheme isomorphism
$$ \pi \colon \Proj(\R(I)) \setminus \pi^{-1}(V(I)) \rt \Spec(A)\setminus \Sing(A).$$
\end{enumerate}
\emph{Example:} Further assume that $A$ is a reduced complete intersection of dimension $d \geq 1$.  Let $I$ be the ideal  as  constructed above. As $A$ is reduced note $\height I > 0$. By construction $\Proj \R(I)$ is non-singular. In particular $\Proj \R(I)$ is a complete intersection.

\section{Proof of  Theorem \ref{rank}}
\s \label{set-rank} \emph{Setup:} In this section $(A,\m)$ is a \CM \ domain. Let $I \subseteq \m$ be a non-zero ideal.  Let $J$ and $\wh{J}$ denote the defining ideals of the Rees algebra (and extended Rees algebra respectively) of $I$.

We first prove
\begin{lemma}
\label{rank-j}(with hypotheses as in \ref{set-rank}) We have
$$ \rank_{\eR(I)} \wh{J}/\wh{J}^2 = \mu(I).$$
\end{lemma}
The proof of Lemma \ref{rank-j} requires a few preparatory  results.
\begin{proposition}\label{rank-one-extn}
(with hypotheses as in \ref{set-rank}). Assume $I = (a_1, \ldots, a_l)$ minimally and $l \geq 2$. Let   $Q = (a_1, \ldots, a_{l-1})$. We have an inclusion $\epsilon \colon \eR(Q) \rt \eR(I)$. Consider the map
$\wt{\epsilon} \colon \eR(Q)[X] \rt \eR(I)$ defined by mapping $X$ to $a_l$. Clearly $\wt{\epsilon}$ is surjective and let $\fp$ be its kernel. Then
$$\rank_{\eR(I)} \fp/\fp^2 = 1.$$
\end{proposition}
\begin{proof}
We note that $\epsilon_{t^{-1}} \colon \eR(Q)_{t^{-1}} \rt \eR(I)_{t^{-1}} $ is an isomorphism. We also note that $t^{-1} \notin \fp$. So we have an exact sequence
\[
0 \rt \fp_{t^{-1}} \rt \eR(Q)[X]_{t^{-1}} = A[t,t^{-1}][X]  \xrightarrow{\wt{\epsilon}_{t^{-1}}} \eR(I)_{t^{-1}} = A[t,t^{-1}] \rt 0.
\]
We have that $f = X - a_lt \in \ker \wt{\epsilon}_{t^{-1}}$.

Claim:  $\fp_{t^{-1}} = (f)$. \\
Let $g(X) \in \fp_{t^{-1}}$. As $f$ is monic we have $g(X) = f q + u$ where $u \in A[t, t^{-1}]$. Then note that $\wt{\epsilon}_{t^{-1}}(u) = \epsilon_{t^{-1}}(u) = 0$. But $\epsilon_{t^{-1}}$ is an isomorphism. So $u = 0$. Thus Claim is proved.

We note that $\fp$ is a height one prime ideal in $\eR(Q)[X]$ and as $t^{-1} \notin \fp$ we have that $\fp_{\fp}$ is principal. So $\eR(Q)[X]_\fp$ is a DVR. The result follows.
\end{proof}
We will also need the following well-known result. We give a proof for the convenience of the reader.
\begin{proposition}
  \label{rank-extn} Let $B \subseteq T$ be Noetherian domains. Let $M$ be a finitely generated $B$-module. Then
  $$\rank_B M = \rank_T M \otimes_B T.$$
\end{proposition}
\begin{proof}
If $M$ is a torsion $B$-module then clearly $M \otimes_B T$ is a torsion $T$-module.
Now assume $\rank_B M = r > 0$. Then there exists an exact sequence $0 \rt B^r \rt M \rt E \rt 0$ where $E$ is a torsion $B$-module. Say $b E = 0$ and $b \in B$ is non-zero.
So we have an exact sequence of $T$-modules
\[
\Tor^B_1(E, T) \xrightarrow{\delta}  T^r \rt M\otimes_B T \rt E\otimes_B T \rt 0.
\]
We note that $b \in B \subseteq T$ annihilates $E\otimes_B T $ and $\Tor^B_1(E, T)$ (and hence $\image \delta$). It follows that $\rank_T M \otimes_B T = r$. The result follows.
\end{proof}
We now give
\begin{proof}[Proof of Lemma \ref{rank-j}]
We prove the result by induction on $\mu(I)$.

We first consider the case when $I = (a)$. As $A$ is a domain we get that $a$ is $A$-regular.
So $\wh{J} = (TX_1 - a)$; see \cite[5.5.9]{HS}. So $S_{\wh{J}}$ is a DVR. In this case the result holds trivially.

Now assume that $\mu(I) = r \geq 2$ and the result holds for all ideals $Q$ such that $\mu(Q) = r -1$. Let $I = (a_1, \ldots, a_r)$ and assume $Q = (a_1, \ldots, a_{r-1})$.
Let $\theta \colon \wh{S} \rt \eR(Q)$ be a minimal presentation. Let $U =\ker \theta$. By induction hypothesis $\rank_{\eR(Q)} U/U^2 = r -1$. Consider $\theta^\prime \colon \wh{S}[X] \rt \eR(Q)[X]$ induced by $\theta$ and mapping $X$ to $X$. We note that $\ker \theta^\prime = U\wt{S}[X]$. Let
$\epsilon \colon \eR(Q) \rt \eR(I)$ be the inclusion and let $\wt{\epsilon}\colon \eR(Q)[X] \rt \eR(I)$ be the map induced by $\epsilon$ and by mapping $X$ to $a_rt$.
We note that $\wt{\epsilon}$ is surjective. Furthermore $\eta = \wt{\epsilon} \circ \theta^\prime \colon \wh{S}[X] \rt \eR(I)$ is a minimal presentation. Let $\wh{J} = \ker \eta$ and let
$\fp = \ker \wt{\epsilon}$.

We apply the Jacobi-Zariski sequence to $\wh{S}[X] \xrightarrow{\theta^\prime} \eR(Q)[X] \xrightarrow{\wt{\epsilon}} \eR(I)$. We have an exact sequence
\begin{equation*}
   D_2(\eR(I)|\eR(Q)[X], \eR(I)) \rt D_1(\eR(Q)[X]|\wh{S}[X], \eR(I)) \rt \wh{J}/\wh{J}^2 \rt \fp/\fp^2 \rt 0. \tag{*}
\end{equation*}
We have $$D_2(\eR(I)|\eR(Q)[X], \eR(I))_{t^{-1}} = D_2(A[t,t^{-1}]|A[t, t^{-1}][X], A[t, t^{-1}]) = 0.$$
So $D_2(\eR(I)|\eR(Q)[X], \eR(I))$ is $t^{-1}$-torsion.
We also have
\begin{align*}
  D_1(\eR(Q)[X], \wh{S}[X])  &= \frac{U\wh{S}[X]}{U^2\wh{S}[X]} \otimes_{\eR(Q)[X]}\eR(I)   \\
   &=  \left(\frac{U}{U^2}\otimes_S S[X] \right)\otimes_{\eR(Q)[X]} \eR(I) \\
   &=  \left(\frac{U}{U^2}\otimes_{\eR(Q)}\eR(Q) \otimes_S S[X] \right) \otimes_{\eR(Q)[X]} \eR(I) \\
  &=  \left(\frac{U}{U^2} \otimes_{\eR(Q)} \eR(Q)[X]  \right) \otimes_{\eR(Q)[X]} \eR(I)\\
  &= \frac{U}{U^2} \otimes_{\eR(Q)} \eR(I).
\end{align*}
By \ref{rank-extn}  it follows that
$$\rank_{\eR(I)}  D_1(\eR(Q)[X], \wh{S}[X]) = \rank_{\eR(Q)}U/U^2 = r -1.$$
By (*) and \ref{rank-one-extn},  it follows that
$$\rank_{\eR(I)} \wh{J}/\wh{J^2} = r.$$
The result follows.
\end{proof}

We now give
\begin{proof}[Proof of Theorem \ref{rank}]
(1) Let $\wh{S} \rt \eR(I)$ be a minimal presentation. By \ref{ss} we have an exact sequence
\[
0 \rt D_2(\eR(I)|\wh{S}, \eR(I)) \rt H_1(\wh{J}) \rt \eR(I)^{\mu(\wh{J})} \rt \wh{J}/\wh{J}^2 \rt 0.
\]
We now $D_2(\eR(I)|\wh{S}, \eR(I))$ is $t^{-1}$-torsion. So by Lemma \ref{rank-j} the result follows.

Claim-1: $\wh{S}_{\wh{J}}$ is a regular local ring.

As $\wh{S}/\wh{J} = \eR(I)$ it follows that $\height \wh{J} = d  + \mu(I) + 1 - (d+1) = \mu(I)$. As $\rank_{\eR(I)} \wh{J}/\wh{J}^2 = \mu(I)$ it follows that $\wh{S}_{\wh{J}}$ is a regular local ring.

(2) We first assert that \\
Claim-2 $S_J$ is a regular ring. \\
We note that $S \subseteq \wh{S} = S[T]$ is a flat extension. We have by \cite[5.5.7]{HS} that $J \subseteq \wh{J} \cap S$. Set $Q = \wh{J} \cap S$. The extension $S_Q \rt \wh{S}_{\wh{J}}$ is flat.
As $\wh{S}_{\wh{J}}$ is a regular local ring it follows that $S_Q$ is a regular local ring. Thus $S_J$ is a regular local ring.

So we have that $\rank_{\R(I)}J/J^2 = \height J$.
As $S/J = \R(I)$ it follows that $\height J = \mu(I) - 1$. By \ref{ss} we have an exact sequence
\[
0 \rt D_2(\R(I)|A, \R(I)) \rt H_1({J}) \rt \R(I)^{\mu(J)} \rt J/J^2 \rt 0.
\]
 Set $X_1 = a_1t$ and localize the above sequence \wrt \ $X_1$. We note that as $\R(I)_{X_1} = \eR(I)_{X_1}$
 $$D_2(\R(I)|A, \R(I))_{X_1} = D_2(\eR(I)|A, \eR(I))_{X_1}.$$
 The later module is $t^{-1}$-torsion. It follows that
 $$ \rank_{R(I)} H_1(J) = \rank_{R(I)_{X_1}} H_1(J)_{X_1} = \mu(J) - \rank J/J^2 = \mu(J) - \mu(I) + 1.$$
\end{proof}

\section{Proof of Proposition \ref{locus}}
In this section we give:
\begin{proof}[Proof of Proposition \ref{locus}]
(I) We note that $S = A[X_1, \ldots, X_n]$ is regular $*$-local ring and $J$ is a graded ideal of $S$. We have $\projdim S/J$ is finite.

If $H_1(J)_P$ is free $\R(I)_P$-module then by a result Gulliksen, see \cite[1.4.9]{GL}, it follows that $J_P$ is generated by a regular sequence. It follows that $\R(I)_P$ is a complete.
intersection.

Conversely if $\R(I)_P$ is a complete intersection then $J_P$ is generated by a regular sequence. It follows that $H_1(J)_P$ is free $\R(I)_P$-module.

(II) This is similar to (I).
\end{proof}

\begin{remark}
  We note that even if we choose minimal generators of $J$ they might not remain minimal in $J_P$
\end{remark}

\section{Proof of Theorem \ref{min-gen}}
In this section we give
\begin{proof}[Proof of Theorem \ref{min-gen}]
Let $Q$ be a minimal reduction of $I$. So $Q = (x_1, \ldots, x_r)$ as $I$ is equi-multiple. We note that $x_1, \ldots, x_r$ is an $A$-regular sequence. Let $\wh{S}_L = A[X_1, \ldots, X_r, T]$
be a minimal presentation of $\eR(Q)$ and let $\eta \colon \wh{S}_Q \rt \eR(Q)$ be the corresponding map. Then by \cite[5.5.9]{HS}  we get that $\ker \eta$ is generated by a regular sequence.
So $D_2(\eR(Q)|A, -) = 0$. Applying the Jacobi-Zariski sequence to $A \rt \eR(Q) \rt \eR(I)$ we obtain an inclusion $D_2(\eR(I)|A, k) \hookrightarrow D_2(\eR(I)|\eR(Q), k)$. It follows that
$\mu(H_1(\wh{J})) \leq \rank_k D_2(\eR(I)|\eR(Q), k)$.

 As $\eR(I)$ is a complete intersection we obtain $D_3(k|\eR(I), k) = 0$. We apply the Jacobi-Zariski sequence to $ \eR(Q) \rt \eR(I) \rt k$. So we get an inclusion $D_2(\eR(I)|\eR(Q), k) \hookrightarrow D_2(k|\eR(Q), k)$.
 We note that as $\eR(Q)$ is a complete intersection $\rank_k D_2(k|\eR(Q), k) = \embdim(\eR(Q))  - \dim \eR(Q)$.
 We note that $\embdim \eR(Q) = 1 + \embdim G_Q(A)$. We have $G_Q(A) \cong A/Q[x_1^*, \ldots, x_r^*]$. So $\embdim G_Q(A) = \embdim{A/Q} + r$.
 By the previous inequality it follows that $\mu(H_1(\wh{J}))  \leq \embdim A/Q + r - d$.
\end{proof}
We given an example which proves that the bound in Theorem \ref{min-gen} is strict.
\begin{example}\label{hyper}
Let $(R,\n)$ be a regular local ring with infinite residue field and let $(A, \m) = (R/(f), \n/(f))$ for some $f \in \n^2$. Take $I = \m$. We first note that $\projdim_{\wh{S}} \eR(\m) = \infty$
(otherwise $\projdim_A \m < \infty$, which is false). So $\wh{J}$ is not generated by a regular sequence. So $H_1(\wh{J}) \neq 0$. So $\mu(H_1(\wh{J})) \geq 1$.
Let $Q$ be a minimal reduction of $\m$. Then notice that $A/Q = \ov{R}/(\ov{f})$ where $\ov{R}$ is a DVR. So $\embdim A/Q = 1$. Thus the bound is attained.
\end{example}
\section{Proof of Theorem \ref{poly-m-thm} and corollary \ref{poly-m-corr-reg}}
\s  \label{setup-poly} Let  $(A,\m)$ be a Noetherian local ring and let $I$ be an $\m$-primary ideal.  Set $\R(I) = \bigoplus_{n \geq 0}I^n$ be the Rees algebra of $I$.  Assume $\dim A = d \geq 1$.
We first prove
\begin{lemma}\label{poly-growth}(with hypotheses as in \ref{setup-poly}) Let $E$ be a finitely generated $\R(\F)$-module. Then
\begin{enumerate}[\rm (1)]
  \item Fix $j \geq 2$. We have $\ell(D_j(\R(I)|A, E)_n$ is finite for all $n \in \Z$.
  \item Fix $j \geq 2$. The function $n \rt \ell(D_j(\R(I)|A, E)_n$ is polynomial of degree $\leq d -1$.
\end{enumerate}
\end{lemma}
\begin{proof}
  (1) Let $\fp \neq \m$ be a prime ideal in $A$. We note that $\R(I)_\fp = A_\fp[t]$ which is a smooth $A_\fp$-algebra. So $D_j(\R(I)|A, E)_\fp = D_j(A_\fp[t]|A_\fp, E_\fp) = 0$. We note that
  $(D_j(\R(I)|A, E)_n)_\fp = (D_j(A_\fp[t]|A_\fp, E_\fp)_n = 0$. We also note that as $D_j(\R(I)|A, E)$ is a finitely generated graded $\R(I)$-module, each $D_j(\R(I)|A, E)_n$ is a finitely generated $A$-module. The result follows.

  (2) As $D_j(\R(I)|A, E)$ is a finitely generated graded $\R(I)$-module, and as by (1) we have $\ell(D_j(\R(I)|A, E)_n$ is finite for all $n \in \Z$ it follows that there exists $s$ such that $\m^s D_j(\R(I)|A, E) = 0$. It follows that $D_j(\R(I)|A, E)$ is a $\R(I)/\m^s\R(I)$-module and the latter has dimension $d$. The result follows.
\end{proof}
We now give
\begin{proof}[Proof of Theorem \ref{poly-m-thm}]

(i) $\implies$ (ii):

We note that $\dim  D_2(\R(I)|A, E) \leq d -i$.  Let $\fp \in \Proj(\R(I)$ with height $\leq i$. As $\R(I)$ is equi-dimensional and catenary it follows that $\dim \eR(I)/\fp = \dim \eR(I) - \height  \fp \geq  d + 1 - i$, see \cite[Section 31,Lemma 2]{Mat} for the local case; the same result holds for $*-local$ case. It follows that $\fp$ is not in the support of $D_2(\R(I)|A, E) $ for any graded $\R(I)$-module $E$. We now choose $E = \R(I)/\fp$ for a prime ideal of height $\leq i$. We have $D_2(\R(I)|A, E) = D_2(\R(I)|S, E)$. Let $Q$ be the inverse image of $\fp $ in $S$. So we have
$$ 0 = D_2(\R(I)|S, \R(I)/\fp)_\fp = D_2(\R(I)_\fp|S_Q, \kappa(\fp)). $$
It follows that $H_1(J_Q) = 0$. The result follows.

(ii) $\implies$ (i): Suppose if possible there exists a finitely generated graded $\R(I)$-module with $\dim D_2(\R(I)|A, E) \geq d -i + 1$. Then there exists a prime ideal $\fp$ of $\R(I)$ in the support of
$D_2(\R(I)|A, E)$ with $\height \fp \leq i$. As $D_2(\R(I)|A, E)$ is a $\R(I)/\m^s$-module  for some $s \geq 1$ it follows  that $\fp \in \Proj(\R(I))$.
But  as $S_Q \rt \R(I)_\fp$ is a complete intersection we have
$$ D_2(\R(I)|S, E)_\fp = D_2(\R(I)_\fp|S_Q, E_\fp) = 0,$$
see \cite[2.5.2]{MR}. This is a contradiction. The result follows.

If (i) (and so (ii))  holds then by \cite[8.4]{I} we have for any prime ideal $\fp$ of $\R(I)$ with $\height \fp \leq i$ we have
$$D_j(\R(I)|S, E)_\fp = D_j(\R(I)_\fp|S_Q, E_\fp) = 0,$$
for any $j \geq 2$. The result follows.
\end{proof}
We now give
\begin{proof}[Proof of Corollary \ref{poly-m-corr-reg}]

(i) $\implies$ (ii):

By the theorem above we obtain that for any prime ideal $\fp$  in $\Proj \R(I)$ of height $\leq i$ the ideal $J_Q$ is a complete intersection. So $H_1(J)_\fp$ is free $\R(I)_\fp$-module.

(ii) $\implies$ (i):

We note that the $S$-ideal $J$ has finite projective dimension. So if $H_1(J)_\fp$ is free then by a result due to Gulliksen \cite[1.4.9]{GL} we obtain that $J_Q$ is a complete intersection.

(i) $\Leftrightarrow$ (iii) This follows from the fact (i) $\Leftrightarrow$ (ii) of the above theorem with the fact that $S$ is regular.
\end{proof}
\section{appendix}
In the appendix we prove  some results which we believe are already known. However we are unable to find a reference for these results. As it is critical for us we give  proofs.

\s\label{cont} We will need the following exercise problem from \cite[Exercise 6.7]{Mat}. Let $A \rt B$ be a homomorphism of Noetherian rings and let $M$ be a finitely generated $B$-module. Then 
$$\Ass_A M  = \{ P \cap A \mid P \in \Ass_B M \}.$$
\begin{theorem}\label{unmixed}
Let $(A,\m)$ be a \CM \ local ring of dimension $d \geq 1$. Let $I \subseteq \m$. Then
\begin{enumerate}[\rm (1)]
  \item $\eR(I)$ is unmixed and equi-dimensional.
  \item If $\height I > 0$ then $\R(I)$ is also unmixed and equi-dimensional.
\end{enumerate}
\end{theorem}
\begin{proof}
(1) Claim-1:  We have $$\Ass \eR(I) = \{ \fp A[t,t^{-1}] \cap \R(I) \mid \fp \in \Ass A \}.$$

Proof of Claim-1: Set $H = \{ \fp A[t,t^{-1}] \cap \R(I) \mid \fp \in \Ass A \}$. As $A$ is \CM \ all associate primes of $A$ are minimal. So by \cite[4.5.5]{BH} $H$ is the set of minimal primes of $\eR(I)$.
Conversely if $Q \in \Ass \eR(I)$ then note that as $t^{-1}$ is $\eR(I)$-regular we have that $t^{-1} \notin Q$. So $Q = \fp A[t, t^{-1}]\cap \eR(I)$ for some prime $\fp$ of $A$. We note that
$\fp = Q\cap A \in \Ass_A \eR(I) = \bigcup_{n \in \Z}\Ass_A  I^n \subseteq \Ass_A A$ (here we are considering $I^n$ as an $A$-module). So Claim-1 follows.

Let  $Q = \fp A[t, t^{-1}]\cap \eR(I)$ be an associate prime of $\eR(I)$. Then $\fp $ is a minimal prime of $A$.
As $A$ is \CM \ we have $\dim A/\fp = d$. So we have a chain of prime ideals in $A$
$$ \fp = P_0 \subseteq P_1 \subsetneq P_2 \subsetneq \cdots \subsetneq P_{d-1} \subsetneq P_d = \m.$$
So we have a chain of prime ideals in $\eR(I)$,
$$ Q = P_0 A[t,t^{-1}] \cap \eR(I) \subsetneq \cdots \subsetneq P_dA[t,t^{-1}]\cap \eR(I) =  \m A[t,t^{-1}]\cap \eR(I).$$
Finally note that $\m A[t,t^{-1}]\cap \eR(I)$ is properly contained in the maximal homogeneous ideal of $\eR(I)$.
So $\dim \eR(I)/Q = d +1$. The result follows.

(2) Let  $Q \in \Ass \R(I)$. Then $Q \cap A = \fp \in \Ass_\A \R(I) = \bigcup_{n \geq 0}\Ass_A  I^n \subseteq \Ass_A A$ (here we are considering $I^n$ as an $A$-module). 
 As $A$ is \CM \ we get that $\fp$ is a minimal prime of $A$ and $\dim A/\fp = d$. As $\height I \geq 1$ thee exists $x \in I$ which is $A$-regular and so not contained in any minimal prime of $A$. We note that $\R(I)_{xt^0} \cong  A_x[t]$. It follows that $Q_{xt^0} = \fp A_x[t]$.
We note that $\dim (A/\fp)_x = d - 1$. So there is a chain of primes in $A_{x}$
$$ \fp = P_0 \subseteq P_1 \subsetneq P_2 \subsetneq \cdots \subsetneq P_{d-1}.$$
This yields a chain of prime ideals in $\R(I)$
\[
 Q = P_0 A_x[t] \cap \R(I) \subsetneq \cdots \subsetneq P_{d-1}A_x[t]\cap \R(I) \subsetneq (P_{d-1} , t)A_x[t]\cap \R(I).
\]
We note that
$$ (P_{d-1}, t)A_x[t]\cap \R(I) \subsetneq  (\m, \R(I)_+).$$
So $\dim \R(I)/Q = d +1$.
It follows that $\R(I)$ is unmixed and equi-dimensional.
\end{proof}

\s \label{sup-elt} Let $(A,\m)$ be a Noetherian local ring and let $I \subseteq \m$ be an ideal in $A$. An element $x \in I$  is said to be $I$-superficial if there exists non-negative integers $c, n_0$ such that $(I^{n+1} \colon x)\cap I^c = I^n$ for all $n \geq n_0$. Superficial elements exist when the residue field $k$ of $A$ is infinite. If $\grade I > 0$ then  it is not difficult to prove that any $I$-superficial element is $A$-regular.

\s \label{const-sup} \emph{Sketch of existence of a superficial element:} Suppose $A/\m$ is infinite. Let $G = G_I(A)$ the associated graded ring of $A$ \wrt \ $I$. Let $P_1, \ldots, P_r, Q_1, \ldots, Q_s$ be the associate primes of $G$ such that $P_i \nsupseteq G_+$ for all $i$ and $Q_j \supseteq G_+$ for all $j$.  Set $P_{i,1} = I/I^2 \cap P_i$. We note that $P_{i.1}$ is properly contained in $I/I^2$. So by Nakayama's lemma  $V_i = (P_{i,1} + \m I)/\m I$ is a proper subspace of the $k$-vector space $U = I/\m I$. As $k$ is infinite it follows that $U \setminus \cup_{i = 1}^{r} V_i$ is non-empty. Then it is not difficult to show that any $x \in I$ such that
$\ov{x} \in U \setminus \cup_{i = 1}^{r} V_i$ is an $I$-superficial element.

We need the following result:
\begin{proposition}\label{generic}
Let $(A,\m)$ be a \CM \ local ring with infinite residue field. Let $I \subseteq \m$ be an ideal of positive height. Then there exists $a_1, \ldots, a_s \in I$ such that
\begin{enumerate}[\rm (1)]
  \item $I = (a_1, \ldots, a_s)$ (minimally).
  \item Each $a_i$ is $A$-regular.
\end{enumerate}
\end{proposition}
To prove this result  we need the following:
\begin{lemma}\label{generic-sup}
Let $(A,\m)$ be a Noetherian local ring with infinite residue field. Let $I \subseteq \m$ be a non-zero ideal. Then there exists $a_1, \ldots, a_s \in I$ such that
\begin{enumerate}[\rm (1)]
  \item $I = (a_1, \ldots, a_s)$ (minimally).
  \item Each $a_i$ is $I$-superficial.
\end{enumerate}
\end{lemma}
We prove Proposition \ref{generic} assuming Lemma \ref{generic-sup}.
\begin{proof}[Proof of Proposition \ref{generic}]
As $A$ is \CM \ we have $\grade Q  = \height Q$ for any ideal $Q$ in $A$. So $\grade I$ is positive. By Lemma \ref{generic-sup} there exists $a_1, \ldots, a_s \in I$ such that
 $I = (a_1, \ldots, a_s)$ (minimally) and
  each $a_i$ is $I$-superficial. As $\grade I > 0$ each $I$-superficial element $a_i$ is $A$-regular
\end{proof}
It remains to give
\begin{proof}[Proof of Lemma \ref{generic-sup}]
  Let $U = I/\m I$. By \ref{const-sup} there exists proper $k$-subspaces $V_i$ of $U$ (with $1\leq i \leq r$) such that if $x \in I$  with
$\ov{x} \in U \setminus \cup_{i = 1}^{r} V_i$ is an $I$-superficial element.

Let $a_1 \in I$ such that $\ov{a_1} \in U \setminus \cup_{i = 1}^{r} V_i$.
If $(a_1) = I$ then we are done. Otherwise $W_1 = ( (a_1) + \m I)/\m I$ is a proper subspace of $U$. Choose $a_2 \in I$ such that $\ov{a_2} \in U \setminus (\cup_{i = 1}^{r} V_i) \cup W_1$.
Then $a_2$ is $I$-superficial and $(a_1) \subsetneq (a_1, a_2) \subseteq I$. Iterating we obtain a sequence
$$(a_1) \subsetneq (a_1, a_2)  \subsetneq \cdots \subsetneq (a_1, \ldots, a_s) \subseteq I$$ such that each $a_i$ is $I$-superficial. If $I = (a_1, \ldots, a_s)$ then we are done.
Otherwise $W_s =  ( (a_1, \ldots, a_s) + \m I)/\m I$ is a proper subspace of $U$. Choose $a_{s+1} \in I$ such that $\ov{a_{s+1}} \in U \setminus (\cup_{i = 1}^{r} V_i) \cup W_s$.
Then $a_{s +1}$ is $I$-superficial and $ (a_1, \ldots, a_s) \subsetneq (a_1, \ldots, a_s, a_{s+1}) \subseteq I$. This process will terminate as $A$ is Noetherian. The result follows.
\end{proof}

\end{document}